\documentclass{amsart}
\usepackage{amsmath}
\usepackage{amsfonts}
\usepackage{amstext}
\usepackage{amsbsy}
\usepackage{amsopn}
\usepackage{amsxtra}
\usepackage{upref}
\usepackage{amsthm}
\usepackage{amsmath}
\usepackage{amssymb}
\usepackage{graphicx}
\usepackage[textures,bookmarks=false]{hyperref}
\usepackage{mathrsfs}

\newtheorem{prop}{Proposition}[section]

\newtheorem{lema}{Lemma}[section]

\newtheorem{eje}{Example}[section]

\newtheorem{introtheorem}{Theorem}

\newtheorem{introcorollary}{Corollary}

\title{The Lyapunov spectrum is not always concave}

\date{December 5, 2008}

\begin{thanks}
{The first author was partially supported by  Proyecto Fondecyt 11070050}
\end{thanks}

\author{Godofredo Iommi} 
\author{Jan Kiwi}
\address{Facultad de Matem\'aticas,
Pontificia Universidad Cat\'olica de Chile, Avenida Vicu\~na Mackenna 4860, Santiago, Chile}
\email{giommi@mat.puc.cl}
\urladdr{http://www.mat.puc.cl/\textasciitilde giommi/}
\address{Facultad de Matem\'aticas,
Pontificia Universidad Cat\'olica de Chile, Avenida Vicu\~na Mackenna 4860, Santiago, Chile}
\thanks{Both authors were partially supported by Research Network on Low Dim
ensional Systems, PBCT/CONICYT, Chile.}
\urladdr{www.mat.puc.cl/\~{}jkiwi}

\begin{document}
%\showlabels

\begin{abstract}
We characterize one-dimensional compact repellers having non-concave Lyapunov spectra. 
For linear maps with two branches we give an explicit condition that characterizes non-concave Lyapunov spectra.
\end{abstract}

\maketitle
\section{Introduction}
The rigorous study of the Lyapunov spectrum finds its roots in the
pioneering work of H. Weiss \cite{w1}.  For the purpose of simplicity
we restrict  our discussion to a  dynamical system
$T : \Lambda \to \Lambda$ where $\Lambda$ is a compact subset of an interval. 
The Lyapunov exponent $\lambda (x)$ of $T$ at $x$ 
is 
$$\lambda (x) := \lim_{n \to \infty} \frac{1}{n} \log |(T^n)'(x)|,$$
whenever this limit exists.
The Lyapunov spectrum $L$ encodes the decomposition of
the phase space $\Lambda$ into level sets $J_\alpha$
of the Lyapunov exponent (i.e. the set where $\lambda(x) = \alpha$).
More precisely, $L$ is the function that
assigns to each $\alpha$ (in an appropriate interval) the Hausdorff 
dimension of $J_\alpha$. 
Weiss proved that the Lyapunov spectrum is real analytic.
A rather surprising result in light of the fact
that this decomposition is fairly complicated. For instance, 
each level set turns out to be dense in $\Lambda$.
Relying
on his previous joint results with Pesin~\cite{pw1}, Weiss studied  the Lyapunov spectrum via the
dimension spectrum of the measure of maximal entropy.
It is worth to mention that not only the analyticity of $L$ is obtained with 
this approach but it also gives, for each $\alpha$, an equilibrium measure  such that $J_\alpha$ has full measure and
the Hausdorff dimension of the measure coincides with that of $J_\alpha$.

In between the wealth of novel and correct results about Lyapunov spectra
briefly summarized above, the aforementioned paper~\cite{w1} unfortunately contains a
claim which is not fully correct. Namely, that compact conformal
repellers have concave Lyapunov spectra~\cite[Theorem 2.4
  (1)]{w1}.  Recently, it has been shown that non-compact conformal repellers may have non-concave Lyapunov spectra  (see
\cite{ks} for the Gauss map and \cite{io} for the Renyi map).  Puzzled
by this phenomena, our aim here is to better understand the concavity
properties of Lyapunov spectra. 
On one hand  we characterize interval maps
with two linear full branches for which the Lyapunov spectra is not
concave. In particular, we not only show that compact conformal repellers (see Example \ref{ex}) may have non-concave spectra,
but that this already occurs in the simplest possible context.  On the other hand, we establish some
general conditions under which the Lyapunov spectra is non-concave.
Roughly speaking, the asymptotic variance (of the $-t \log |T'|$ potential)
should be sufficiently large (in a certain sense).

The class of maps that we consider is defined as follows.
Given a pairwise disjoint finite family of closed intervals $I_1, \dots, I_n$
contained in $[0,1]$ we say that a map:
$$T: \bigcup_{i=1}^n I_i \to [0,1]$$
is a {\sf cookie-cutter map with $n$ branches} if the following holds:
\begin{enumerate}
\item $T(I_i)= [0,1]$ for every $i \in \{1 , \dots, n\}$,
\item The map $T$ is of class $C^{1 + \epsilon}$ for some $\epsilon >0$,
\item $|T' (x)| > 1$ for every $x \in I_1 \cup \cdots \cup I_n $.
\end{enumerate}
We say that $T$ is a {\sf linear cookie-cutter map} if $T$ restricted to each one of the intervals
$I_i$ is an affine map. The {\sf repeller} $\Lambda \subset [0,1]$  of $T$ is 
\[ \Lambda:= \bigcap_{n=0}^{\infty} T^{-n} ([0,1]).      \]
The {\sf Lyapunov exponent} of the map $T$ at the point $x \in [0,1]$ is defined by
\[  \lambda(x) = \lim_{n \to \infty} \frac{1}{n} \log |(T^n)'(x)|, \]
whenever the limit exists.  
Let us stress that the set of points for which the Lyapunov exponent does not exist has full Hausdorff dimension~\cite{bs}.

We will mainly be concerned with the Hausdorff dimension of the level sets of $\lambda$.
More precisely, the range of $\lambda$ is an interval $[\alpha_{min}, \alpha_{\max}]$ and
the {\sf multifractal spectrum of the Lyapunov exponent}  is the function given by:
$$\begin{array}{rccc}
  L:& [\alpha_{min}, \alpha_{\max}] & \rightarrow & \mathbb{R} \\
    & \alpha & \mapsto & \dim_H ( J_\alpha = \{ x \in \Lambda \mid \lambda(x) = \alpha \}),
\end{array}$$
where $\dim_H(J)$ denotes the Hausdorff dimension of $J$.
For short we say that $L$ is the {\sf Lyapunov spectrum of $T$}.

In our first result we consider linear cookie-cutters with two branches and obtain conditions on the slopes that ensure that the Lyapunov spectrum is concave. This result can also be used to construct examples of non-concave Lyapunov spectrum.

\begin{introtheorem}
  \label{A}
  Consider the linear cookie-cutter map with two branches 
$$T:  \left[0, \frac{1}{a} \right] \cup \left[1-\frac{1}{b}, 1\right]  \to  \left[0,1\right] $$
 defined by
  $$T(x) =  \begin{cases} 
                                 a x & \text{ if } x \le \frac{1}{a}, \\
                                 b x + 1 - b & \text{ if } x \ge 1-\frac{1}{b}.
\end{cases} $$
  Then the Lyapunov spectrum $L:[\log a , \log b] \to \mathbb{R}$ of $T$ is concave if and only if 
$$\dfrac{\log b}{\log a} \leq \dfrac{\sqrt{2 \log 2} + 1}{\sqrt{2 \log 2} - 1} \approxeq 12.2733202...$$
\end{introtheorem}

The above relation gives  the bifurcation
point dividing the spectra with inflection points form the concave
ones. For maps with two linear branches, the combination of Lemma~\ref{tran} with this Theorem implies that the
bifurcation between concave spectra and non-concave one may only occur when
the Lyapunov exponent, $\alpha_M$, corresponding to the measure of
maximal entropy is an inflection point.

In order to describe the Lyapunov spectrum we will make use of the thermodynamic formalism. Let $T$ be a cookie-cutter map, denote by $ \mathcal{M}_T$ the set of          
 $T-$invariant probability measures.  The {\sf topological pressure} of $-t \log |T'|$ with respect to $T$ is defined by
\begin{equation*}
P(-t \log |T'|) = \sup \left\{ h(\mu) -t \int \log |T'| \ d\mu :  \mu \in \mathcal{M}_T   \right\},
\end{equation*}
where $ h(\mu)$ denotes the measure theoretic entropy of $T$ with respect to the measure $\mu$ (see \cite[Chapter 4]{wa} for a precise definition of entropy). A measure $\mu_{t} \in  \mathcal{M}_T$ is called an {\sf equilibrium measure} for $-t \log |T'|$ if it satisfies:
\[ P(-t \log |T'|) = h(\mu_{t}) - t \int \log |T'| \ d\mu_{t}. \]
If the function $\log |T'|$ is not cohomologous to a constant then the function $t \mapsto P(-t \log |T'|)$ is strictly convex, strictly decreasing, real analytic and for every $t \in \mathbb{R}$ there exists a unique equilibrium measure $\mu_t$ corresponding to $-t \log |T'|$ (see \cite[Chapters 3 and 4]{pp}). Moreover, there are explicit expressions for the derivatives of the pressure. Indeed (see \cite[Chapter $4$]{pu}), the first derivative of the pressure is given by
\[\alpha(t_0):= -\frac{d}{dt} P(-t \log |T'|) \Big|_{t=t_0} =  \int \log |T'| \ d \mu_{t_0}.\]
The second derivative of the pressure is the {\sf asymptotic variance}
\[\frac{d^2}{dt^2} P(-t \log |T'|) \Big|_{t=t_0} = \sigma^2(t_0),\]
where
\[\sigma^2(t_0):=
\lim_{n \to \infty} \int \left( \sum_{i=0}^{n-1} - \log |T'(T^i x)| +n
\int \log |T'| d \mu_{t_0} \right) ^2 d\mu_{t_0}(x). \] 
There exists a
close relation between the topological pressure and the Hausdorff
dimension of the repeller. In fact, the number $t_d = \dim_H(\Lambda)$
is the unique zero of the Bowen equation (see \cite[Chapter $7$]{pe})
\[P(-t \log |T'|) =0.\] 
Let $\mu_{t_d}$ be the unique equilibrium measure corresponding to
the function $-t_d \log |T'|$ and, let
\[\alpha_d := \int \log |T'| \ d\mu_d.\]
Our next theorem establishes general conditions for a cookie-cutter map to have (non-)concave Lyapunov spectrum.
\begin{introtheorem}
  \label{B}
  Let $L$ be  Lyapunov spectrum of  a cookie-cutter map $T$.  Then $L$ is always concave in $[\alpha_{min} , \alpha_d]$. 
Moreover, $L: [\alpha_{min}, \alpha_{max}] \to \mathbb{R}$ is concave if and only if
$${\sigma^2(t)} <   \dfrac{\alpha(t)^2} {2 P(-t \log |T'|)} \quad \text{ for all } \quad t < t_d \cdot$$ 
\end{introtheorem}
When considering linear cookie-cutter maps we obtain a simpler formula in terms of the slopes of the map.

\begin{introcorollary}
  \label{C}
  Consider a linear cookie-cutter map $T$ with  $n$-branches of slopes $m_1, \dots, m_n$.
  Then its  Lypunov spectrum $L: [\min\{\log|m_i|\}, \max\{\log|m_i|\}] \to \mathbb{R}$ 
is concave if and only if, for all $t \in \mathbb{R}$, 
\begin{equation*}
2 \log \left( \sum_{i=1}^n |m_i|^t \right) \left(\frac{\left( \sum_{i=1}^n |m_i|^t (\log |m_i|)^2  \right) \left( \sum_{i=1}^n |m_i|^t   \right)}{\left( \sum_{i=1}^n |m_i|^t \log |m_i|  \right)^2}  - 1 \right) \leq 1.
\end{equation*}
\end{introcorollary}

In terms of the $L^1(\mu_t)$ and $L^2(\mu_t)$ norms with respect to the corresponding equilibrium measure $\mu_t$ the above formula
may be rewritten as:
$$
2 P(-t \log |T'|) \left( \dfrac{ \| \log|T'| \|^2_{2,t}}{ \| \log|T'| \|^2_{1,t} } -1  \right)  \le 1.$$

\medskip
Although our results shed some light on the concavity properties of the Lyapunov spectrum, up to our knowledge, the occurrence of inflection points is
 not well understood.
In fact, given a map with Lyapunov spectrum having an inflection point at $\alpha_0$, 
it is natural to pose the general problem of understanding the geometric and ergodic properties of the equilibrium measure
with exponent $\alpha_0$.

From our results it follows  that the Lyapunov spectrum of a cookie-cutter map has an even (possibly zero) number of inflections points. 
Although Theorem~\ref{A} establishes the existence of maps with spectra having at least two inflection points, it does not give 
an exact count. We conjecture that the spectrum  of a cookie-cutter map with two branches has at most two inflection points.
Also one may ask: Is there an upper bound on the number of inflection points of the spectrum of a cookie-cutter map? of a compact conformal repeller? of a non-compact conformal repeller?

\bigskip
Our results are based on the following formula that ties up the topological pressure with the Lyapunov spectrum 
\begin{equation}
\label{WeissFormula}
L(\alpha) = \frac{1}{\alpha} \inf_{t \in \mathbb{R}} (P(-t \log |T'|) +t \alpha).
\end{equation}
This formula follows form the work of Weiss \cite{w1} and can be found explicitly, for instance, in the work of Kesseb\"ohmer and Stratmann \cite{ks}. Actually, in this setting, the Lyapunov spectrum can be written as
\begin{equation} 
\label{lyapunov}
L(\alpha)=  \frac{1}{\alpha} (P(-t_{\alpha} \log |T'|) +t_{\alpha} \alpha) = \frac{h(\mu_{\alpha})}{\alpha},
\end{equation}
where $t_{\alpha}$ is the unique real number such that 
\[ -\dfrac{d}{d t} P(-t \log |T'|) \Big|_{t=t_{\alpha}} = \int \log |T'| \ d\mu_{\alpha} = \alpha, \]
and $\mu_{\alpha}$ is the unique equilibrium measure corresponding to the potential  $-t_{\alpha} \log |T'|$. 
That is, $\alpha \mapsto t_\alpha$ is the inverse of $t \mapsto \alpha(t)$.
Thus, after the substitution $\alpha = \alpha(t)$, equation~\eqref{WeissFormula} becomes:
\begin{equation} 
\label{lt}
L(\alpha(t))=  \frac{1}{\alpha(t)} \left( P(-t  \log |T'|)  +t \alpha(t) \right).  
\end{equation}

The structure of the paper is as follows. In Section~\ref{2} we prove Theorem~\ref{A} and construct explicit examples of maps for which the Lyapunov spectrum is not concave. In Section~\ref{nl} we prove Theorem~\ref{B} and finally is Section~\ref{lc} we prove Corollary~\ref{C}.

\section{Maps with two branches} \label{2}
Our aim now is to prove Theorem~\ref{A}.
Throughout this section we let $T$ be the cookie-cutter with two linear branches of slopes $b>a>1$ and Lyapunov spectrum $L$,
as in the statement of Theorem~\ref{A}.
The proof relies on explicit formulas for $L$ together with a characterization of the inflection points that persist  under small changes
of the slopes $a$ and $b$ (i.e. transversal). 
We show that unstable zeros of $\frac{d^2 L}{d \alpha^2} (\alpha)$ may only occur at the Lyapunov exponent of the measure of maximal entropy when the logarithmic ratio of
the slopes is as in the statement of the theorem. 

We start by obtaining an explicit formula for $L$, from equation~\eqref{WeissFormula}.
\begin{lema} \label{2BranchFormula}
  \begin{equation*} 
L(\alpha) = 
 \frac{1}{\alpha} \left[- \left( \frac{\log b - \alpha}{\log \frac{b}{a}} \right) \log  \left( \frac{\log b - \alpha}{\log \frac{b}{a}} \right) -  \left( \frac{\alpha - \log a }{\log \frac{b}{a}} \right) \log  \left( \frac{\alpha - \log a}{\log \frac{b}{a}} \right) \right].
\end{equation*}
\end{lema}

\begin{proof}
 The equilibrium  measure $\mu_{\alpha}$, in equation~\eqref{lyapunov}, can be explicitly determined. This is due to the fact that it is a Bernoulli measure (see \cite[Theorem 9.16]{wa}). Since the Lyapunov exponent corresponding to $\mu_{\alpha}$ is $\alpha$ we have that
\begin{eqnarray*}
\alpha &=&  \int \log |T'| \ d\mu_{\alpha} = \mu_{\alpha}(I_1) \log a  + \mu_{\alpha}(I_2) \log b 
\\     &=&    \mu_{\alpha}(I_1) \log a +(1- \mu_{\alpha}(I_1)) \log b \\
       &=&  \mu_{\alpha}(I_1) \left( \log a - \log b \right) + \log b.
\end{eqnarray*}
Hence,
\begin{eqnarray*}
\mu_{\alpha}(I_1)& =& \frac{\log b - \alpha}{\log b - \log a}, \\
\mu_{\alpha}(I_2) & = & \frac{\alpha - \log a}{\log b - \log a}.  
\end{eqnarray*}
Therefore, 
$\mu_{\alpha}$ is the unique Bernoulli measure which satisfies the above conditions.  Moreover, the entropy of this measure is
\begin{equation} \label{entropy}
h(\mu_{\alpha})=-\left( \frac{\log b - \alpha}{\log b - \log a} \right) \log  \left( \frac{ \log b - \alpha}{\log b - \log a} \right) -  \left( \frac{\alpha - \log a}{\log b - \log a} \right) \log  \left( \frac{\alpha - \log a}{\log b - \log a} \right). 
\end{equation}
Hence, from equation \eqref{lyapunov} we obtain that
\begin{equation*} \label{dosramas}
L(\alpha) = 
- \frac{1}{\alpha \log \frac{b}{a}} \left[ \left({\log b - \alpha} \right) \log  \left( \frac{\log b - \alpha}{\log b - \log a} \right) +  \left({\alpha - \log a } \right) \log  \left( \frac{\alpha - \log a}{\log b - \log a} \right) \right].
\end{equation*}
\end{proof}

As suggested by equation~\eqref{lyapunov} the behavior of the entropy function is closely related to the (in)existence of inflection points of the Lyapunov spectrum. In fact:
\begin{lema}[Remark 8.1 \cite{io}] \label{et}
A point $\alpha_0 \in (\log a , \log b)$ satisfies $ \frac{d^2L}{d \alpha^2} (\alpha_0)=0$ if and only if
\[ 2 \frac{d L}{d \alpha} (\alpha_0) =  \frac{d^2}{d \alpha^2} h(\mu_{\alpha}) \Big{|}_{\alpha=\alpha_0}. \]
\end{lema}

Thus, our aim now is to study the second derivative of the entropy:

\begin{prop} \label{segunda}
Let $\alpha_M= (\log a+ \log b)/2$. Then
the function $$\frac{d^2}{d \alpha^2} h(\mu_{\alpha}) :(\log a , \log b) \to \mathbb{R}$$ 
is concave, increasing in the interval $(\log a ,\alpha_M)$ and, decreasing
in the interval $(\alpha_M , \log b)$. In particular, it has a unique maximum at 
$\alpha = \alpha_M$. Moreover,
\begin{eqnarray*}
\frac{d^2}{d \alpha^2} h(\mu_{\alpha}) \Big|_{\alpha = \alpha_M} & = & 
- \Big(\frac{2}{\log \frac{b}{a}} \Big)^2, \\
\lim_{ \alpha \to \log a}\frac{d^2}{d \alpha^2} h(\mu_{\alpha}) &=& - \infty,\\
\lim_{ \alpha \to \log b}  \frac{d^2}{d \alpha^2} h(\mu_{\alpha}) & = &- \infty.
\end{eqnarray*}
\end{prop}

\begin{proof}
From equation \eqref{entropy} we have that the first derivative of the entropy with respect to $\alpha$ is given by
\begin{equation} \label{1-derivative}
\frac{d}{d \alpha} h(\mu_{\alpha})=\frac{1}{\log \frac{b}{a}} \left(  \log  \left( \frac{\alpha - \log a}{\log {b} - \log {a}} \right)    - \log  \left( \frac{\log b - \alpha}{\log {b}- \log{a}} \right) \right). 
\end{equation}
We can also compute its second derivative,
\begin{eqnarray} \label{2-derivative}
\frac{d^2}{d \alpha^2} h(\mu_{\alpha})=\frac{-1}{\log \frac{b}{a}} \left( 
\frac{1}{\alpha - \log a} +\frac{1}{\log b - \alpha} \right).
 \end{eqnarray}
Note  that this function has two asymptotes at $\log a$ and at $\log b$.
Indeed
\[ \lim_{ \alpha \to \log a} \frac{d^2}{d \alpha^2} h(\mu_{\alpha}) = - \infty \textrm{ and } \lim_{ \alpha \to \log b} \frac{d^2}{d \alpha^2} h(\mu_{\alpha})= - \infty. \]   
In order to obtain the maximum of $\frac{d^2}{d \alpha^2} h(\mu_{\alpha})$ we compute the third derivative of the entropy function,
 \[ \frac{d^3}{d \alpha^3} h(\mu_{\alpha})=    \frac{1}{\log \frac{b}{a}} 
\left( \frac{1}{(\alpha - \log a)^2} -\frac{1}{(\log b- \alpha)^2}    \right), \]
which is equal to zero if and only if
\[ \alpha=\alpha_M=  \frac{\log b + \log a}{2} . \]
%Note that
%\[D^2\left( \frac{\log b + \log a}{2}   \right) =  - \left( \frac{2}{\log \frac{b}{a}}   \right)^2 <0 \]
Moreover,
\[ \frac{d^4}{d \alpha^4} h(\mu_{\alpha})=    \frac{-2}{\log \frac{b}{a}} 
\left( \frac{1}{(\alpha - \log a)^3} +\frac{1}{(\log b- \alpha)^3}    \right), \]
In particular
\[  \frac{d^4}{d \alpha^4} h(\mu_{\alpha}) \leq 0.  \]
\end{proof}

Now we combine the Proposition \ref{segunda} with Lemma~\ref{et} in order
to characterize inflection points of the spectrum which are stable under small changes of the slopes $a$ and $b$.
Namely, transversal intersections between the graphs of  $2  \frac{d L}{d \alpha} (\alpha)$ and
 $\frac{d^2}{d \alpha^2} h(\mu_{\alpha})$.

\begin{lema} \label{tran}
Assume that $\alpha_i \in (\log a, \log b)$ with $\alpha_i \neq \alpha_M$ is such that $$2  \frac{d L}{d \alpha} (\alpha_i) = 
\frac{d^2}{d \alpha^2} h(\mu_{\alpha}) \Big|_{\alpha = \alpha_i}.$$
Then the intersection of the graphs of $ \frac{d L}{d \alpha} (\alpha)$ and 
$ \frac{d^2}{d \alpha^2} h(\mu_{\alpha}) $ at $\alpha = \alpha_i$ is transversal.
\end{lema}

\begin{proof}
By contradiction, suppose  that the intersection is not transversal, that is
\begin{equation} \label{trans}
 2  \frac{d^2 L}{d \alpha^2} (\alpha_i) =  \frac{d^3 h(\mu_\alpha)}{d \alpha^3} (\alpha_i).
 \end{equation} 
Since $$2 \frac{d}{d \alpha} L(\alpha_i) =  \frac{d^2}{d \alpha^2}h (\alpha_i)$$ we have that $ \frac{d^2 L}{d \alpha^2} (\alpha_i)=0$. Therefore
equation \eqref{trans} implies that $  \frac{d}{d \alpha}h (\alpha)\Big|_{\alpha =\alpha_i}=0$. Thus, by Proposition \ref{segunda},
$\alpha_i= \alpha_M$. 
\end{proof}

Finally, we are ready to prove the Theorem.

\begin{proof}[Proof of Theorem~\ref{A}]
From equation \eqref{lyapunov} we obtain that
\begin{eqnarray*}
2 \frac{d L}{d \alpha}(\alpha)= \frac{2}{\alpha^2 \log \frac{b}{a} } \left[ \log b \log
  \left( \frac{\log b - \alpha}{\log {b} - \log {a}} \right) - \log a
  \log \left( \frac{\alpha - \log a}{\log b - \log{a}} \right)
  \right].
 \end{eqnarray*} 
Therefore,
\begin{eqnarray*}
2 \frac{d L}{d \alpha} (\alpha) \Big|_{\alpha=\alpha_M} &=& -\frac{8 \log 2}{ (\log a + \log b)^2}.
\end{eqnarray*}
Moreover,
\begin{eqnarray*}
\lim_{\alpha \to \log a} 2\frac{d L}{d \alpha}(\alpha) & = & \infty, \\ 
\lim_{\alpha \to \log b} 2\frac{d L'}{d \alpha} (\alpha) & = &  -\infty =  \lim_{\alpha \to \log b} \left( 2\frac{d L}{d \alpha}(\alpha) -  \frac{d^2}{d \alpha^2} h(\mu_{\alpha}) \right).
\end{eqnarray*}
Hence, a sufficient 
condition to have  two transversal
intersections of the graphs of $2\frac{d L}{d \alpha}$ and $\frac{d^2}{d \alpha^2}h$ 
is  $2\frac{d L}{d \alpha}  (\alpha_M) >  \frac{d^2}{d \alpha^2}h(\alpha_M)$.
Now, 
$$2 \frac{d L}{d \alpha}  (\alpha_M) \ge  \frac{d^2}{d \alpha^2}h(\alpha_M) \iff \dfrac{\log b}{\log a} \geq \dfrac{\sqrt{2 \log 2} + 1}{\sqrt{2 \log 2} - 1},$$
where equality holds in one equation if and only if it holds in the other.
To finish the proof of the theorem we must check that there exist values of
$a$ and $b$ such that 
$$\dfrac{\log b}{\log a} < \dfrac{\sqrt{2 \log 2} + 1}{\sqrt{2 \log 2} - 1},$$
for which the corresponding graphs of $2\frac{d L}{d \alpha}$ and $\frac{d^2}{d \alpha}h$ do not intersect.
For this purpose let $a >1$ and $b = \exp(\varepsilon) a$. To ease notation
we introduce $$x = \dfrac{\alpha-\log a}{\log b - \log a},$$
and note that $0 \leq x \leq 1$.
It follows that
$$ \varepsilon^2 x (1 -x) 2 \frac{d L}{d \alpha} (\alpha) = \dfrac{2 \varepsilon}{\alpha^2} g(x),$$
where $g:[0,1] \rightarrow \mathbb{R}$ is the continuous function such that
for all $x \in (0,1)$,
$$g(x) = \varepsilon x (1-x) \log (1-x) + (\log a) x (1-x) \log (1-x) - 
(\log a) (1-x) x \log x.$$ 
Hence, uniformly for $\alpha \in [\log a, \log b]$, we have that
$ \varepsilon^2 x (1 -x) 2 \frac{d}{d \alpha}L(\alpha) \to 0$, as $\varepsilon \searrow 0$.
However, from equation~\eqref{2-derivative} 
$$ \varepsilon^2 x (1 -x) \frac{d^2}{d \alpha^2} h(\mu_{\alpha}) = -1.$$
Thus, for $\varepsilon >0$ sufficiently small, we have that $2\frac{d L}{d \alpha} (\alpha) > \frac{d^2}{d \alpha^2}h(\alpha)$ 
for all $\alpha \in [\log a, \log b]$. 
\end{proof}

\begin{eje}[Non-concave Lyapunov spectrum] \label{ex}
{\em Let $T$ be a linear cookie-cutter map with two branches of slopes
$a=\exp(1)$ and $b= \exp(45)$. The corresponding Lyapunov spectrum is not concave since
\[45=\dfrac{\log b}{\log a} > \dfrac{\sqrt{2 \log 2} + 1}{\sqrt{2 \log 2} - 1}. \]
 See Figure~\ref{2BranchFigure} (right).}
\end{eje}

\begin{figure} 

\centerline{
\includegraphics[width=4cm]{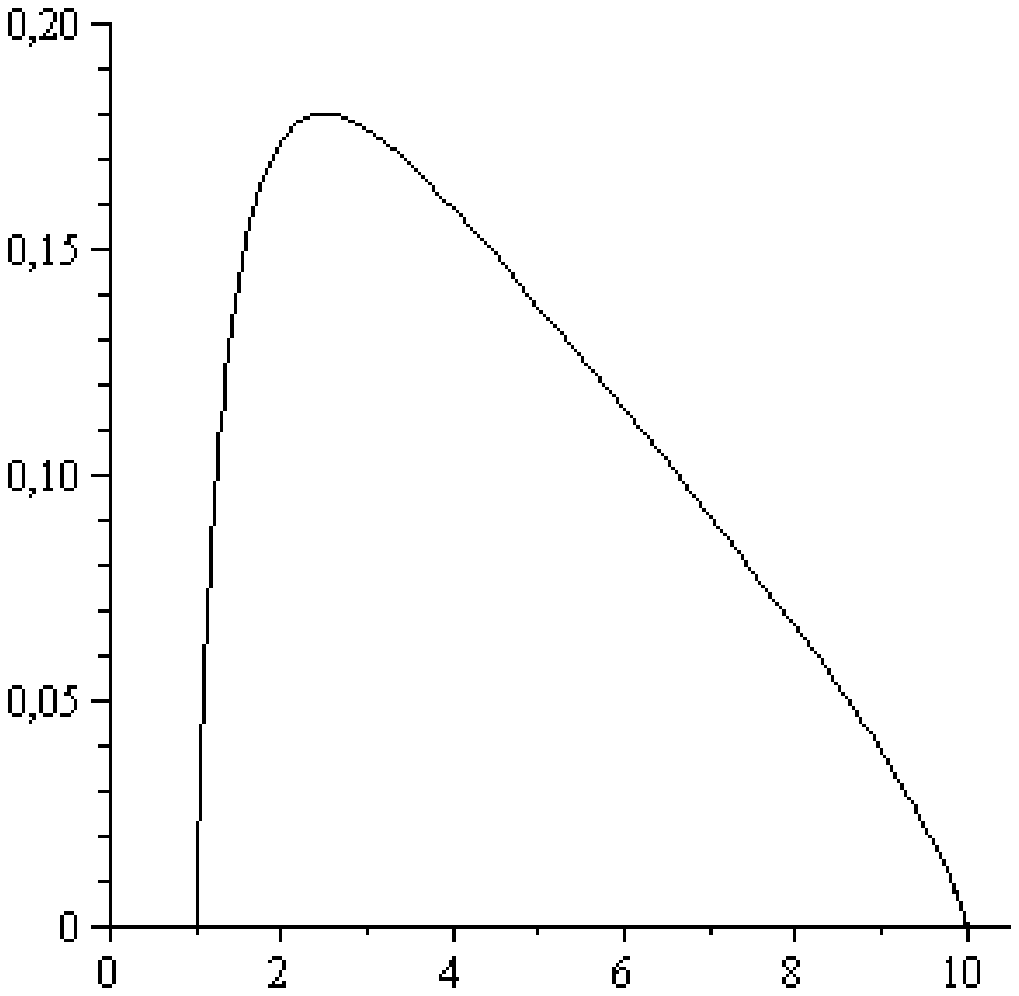}
\includegraphics[width=4cm]{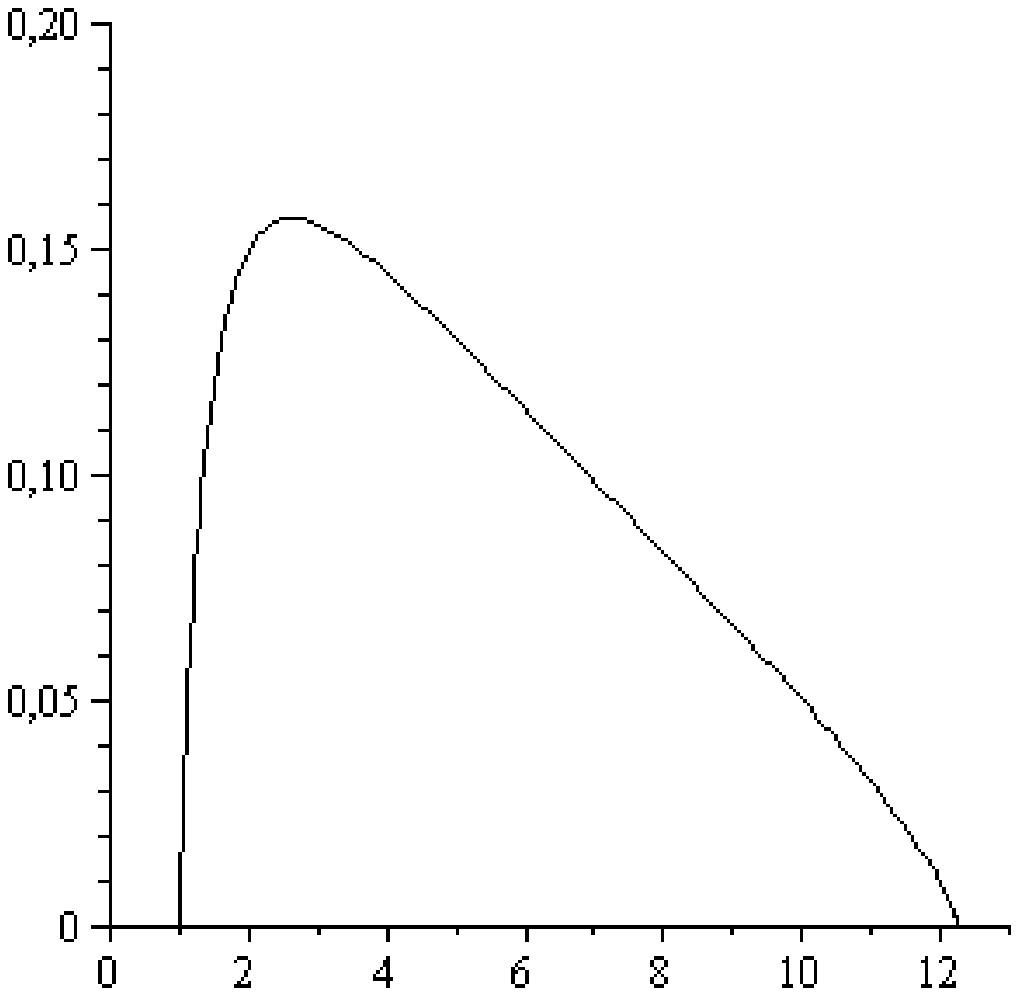}
\includegraphics[width=4cm]{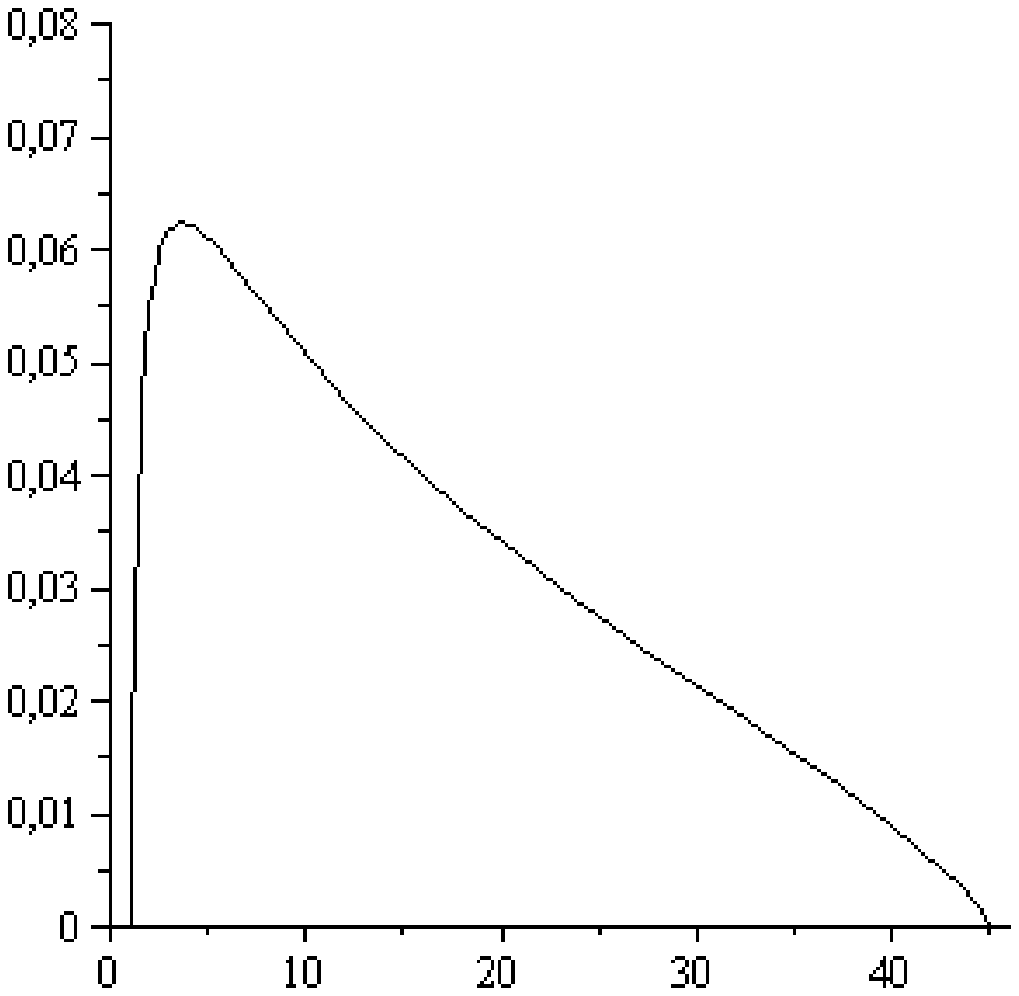}}
\caption{The Lyapunov spectra for maps with two linear branches of slopes $a$ and $b$. For all
graphs $a=\exp(1)$. At the left a concave spectra corresponding to  $b= \exp(10)$. At the center, also concave
but $b \approx (\sqrt{2 \log 2} + 1) / (\sqrt{2 \log 2} - 1)$ is the bifurcation value of the slope. At the right,
the non-concave graph corresponding to  $b=\exp(45)$. \label{2BranchFigure}}
\end{figure}

\section{The non-linear case} \label{nl}
In this section we prove Theorem~\ref{B}, which gives 
a general condition that ensures the existence of inflection points for the Lyapunov spectrum. 
In the next section we will show that the mentioned condition is explicit for maps with linear branches.

\begin{proof}[Proof of Theorem~\ref{B}]
From equation~\eqref{lt} we have
\[L(\alpha(t))= \frac{P(-t \log |T'|) +t \alpha(t)}{\alpha(t)}.\]
Then, using $f(g(t))'$ to denote the derivative of $f \circ g$ with respect to $t$, we have
\[L(\alpha(t))'= \frac{1}{\alpha(t)^2} \left( - \alpha'(t) P(-t \log |T'|) + \alpha(t) P(-t\log |T'|)' +\alpha(t)^2     \right).\]
Recall that 
\begin{equation} \label{rela}
\alpha(t) = -P(-t \log |T'|)'.
\end{equation}
 Hence
	\[L(\alpha(t))' = -\frac{\alpha'(t) P(-t \log |T'|)}{ \alpha(t)^2}. \]
Therefore, making use of equation \eqref{rela} we obtain that
\begin{eqnarray*}
\frac{d^2 L}{d \alpha^2} (\alpha(t)) 
& =& \frac{1}{\alpha'(t)} \cdot \left( \frac{L(\alpha(t))'}{\alpha'(t)} \right)' \\
& = & \frac{1}{\alpha'(t)} \cdot \left( \frac{-P(-t \log |T'|)}{\alpha(t)^2}\right)' \\ 
& =& \frac{1}{\alpha'(t)} \cdot \frac{2\alpha'(t) P(-t \log |T'|) - \alpha(t) P(-t \log |T'|)'}{\alpha(t)^3}\\ 
& =& \frac{1}{\alpha'(t)} \cdot \frac{-2 \alpha'(t) P(-t \log |T'|) + \alpha(t)^2}{\alpha(t)^3} \cdot\end{eqnarray*}
Since 
$$\alpha' (t) = \sigma^2 (t) = P(-t \log |T'|)'' < 0$$
for all $t \in \mathbb{R}$, we conclude that
\begin{equation}
\label{sindividir}
\frac{d^2 L}{d \alpha^2} (\alpha(t)) \le 0 \iff -2 \sigma^2(t)  P(-t \log |T'|) + \alpha(t)^2 \ge 0.
\end{equation}
The theorem follows since $P(-t \log |T'|)$ is positive if and only if $t < t_d$.
\end{proof}

Let us stress that even though Theorem \ref{B} is very general, it is not easy 
to deduce explicit conditions on $T$ in order to guarantee 
the absence or presence of inflection points (in contrast with Theorem \ref{A}).

\section{The linear case} \label{lc}
Throughout this section we consider a linear cookie-cutter map with $n$ branches of slopes $m_1, \dots, m_n$.
We record a straightforward computation in the next lemma. 
This trivial calculation allow us to ''effectively'' draw the graphs of the corresponding Lyapunov spectra (e.g. see Figure~\ref{3BranchFigure} ).

\begin{lema}
  \label{graph}
Consider a linear cookie-cutter  map $T$ with $n$ branches of  slopes $m_1, \dots, m_n$.
  Then,
  \begin{eqnarray*}
    \alpha(t) &= & \frac{\sum_{i=1}^n  |m_i|^t \log |m_i|}{\sum_{i=1}^n |m_i| ^t},\\
    L(\alpha(t)) &=& \frac{\left(\sum_{i=1}^n |m_i|^t \right) \log \left(\sum_{i=1}^n |m_i|^t \right)}{\sum_{i=1}^n | m_i|^t \log |m_i|} -t.
  \end{eqnarray*}
\end{lema}

Note that the graph of $L: [\min\{\log|m_i|\}, \max\{\log|m_i|\}] \to \mathbb{R}$ 
coincides with the graph of $\mathbb{R} \ni t \mapsto (\alpha(t), L(t))$.

\begin{figure} 
\centerline{
\includegraphics[width=4cm]{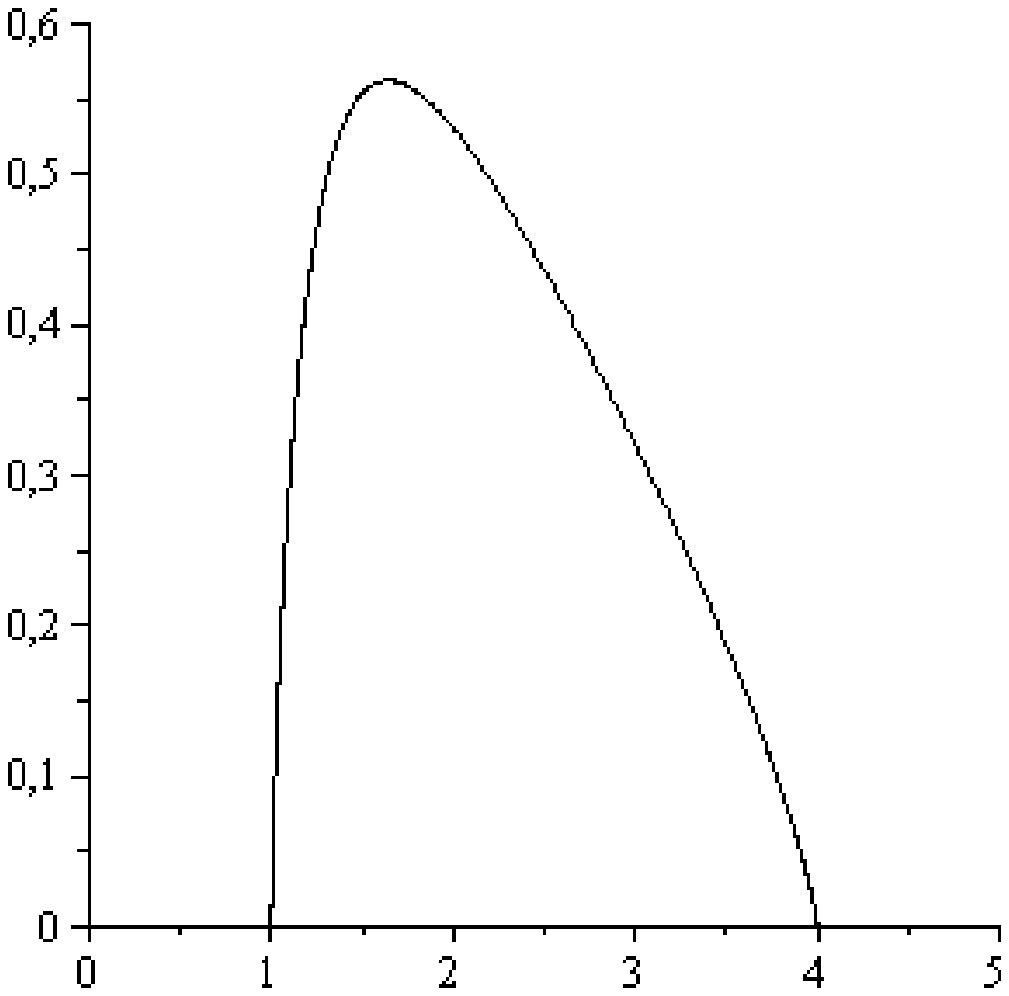}
\includegraphics[width=4cm]{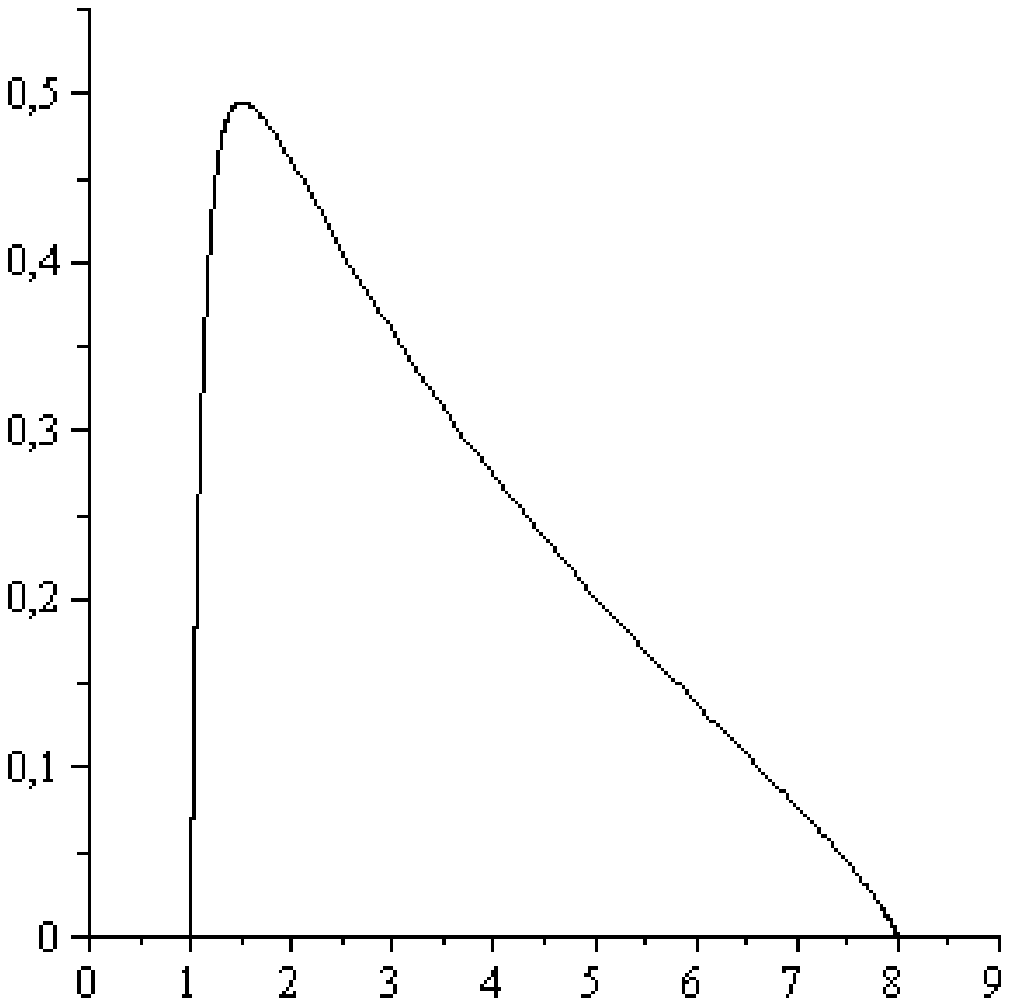}
\includegraphics[width=4cm]{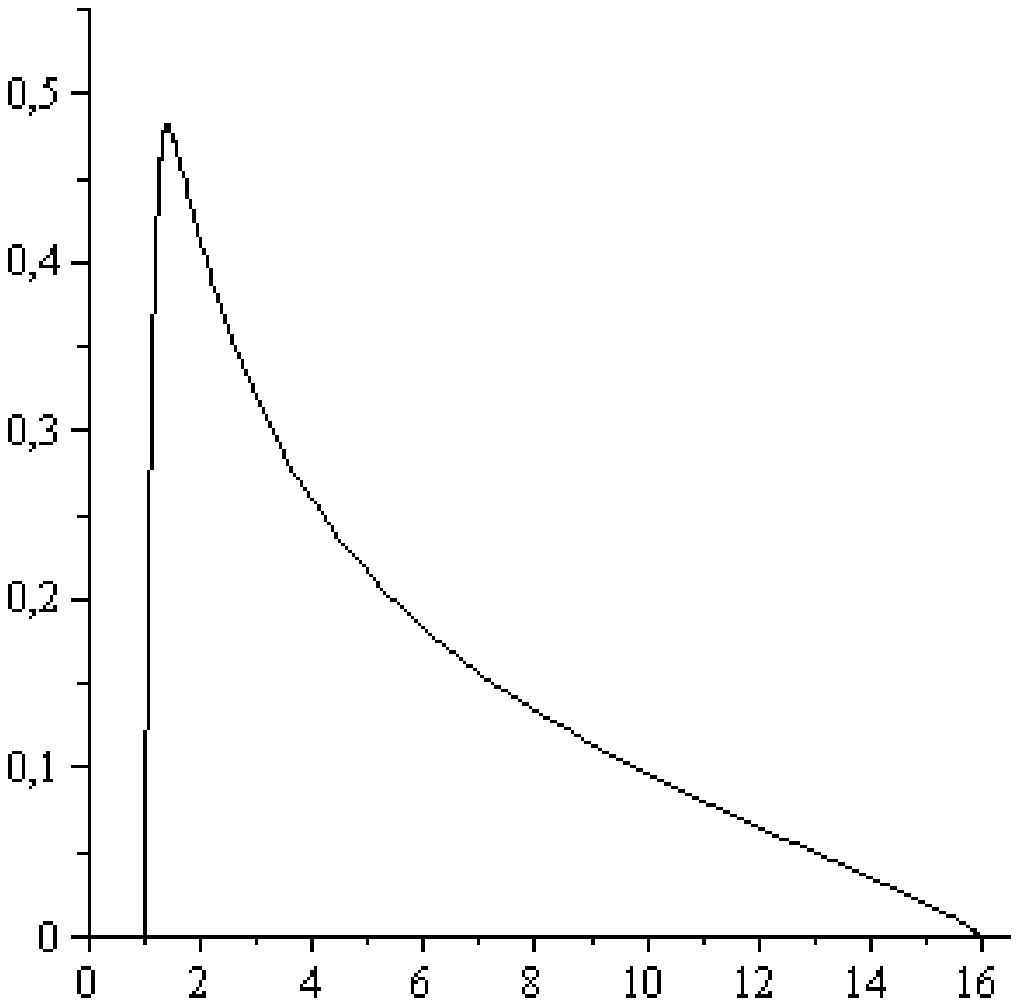}}
\caption{The Lyapunov spectra for maps with three linear branches of slopes $a =\exp(1) < b=\exp(2) < c$.  
At the left a concave spectra corresponding to  $c= \exp(4)$. At the center, a non-concave
corresponding to  $c= \exp(8)$. At the right,
another non-concave spectrum corresponding to  $b=\exp(16)$. \label{3BranchFigure}}
\end{figure}

\begin{proof}
The pressure function of a linear cookie-cutter map has a simple form (see, for instance,  \cite[Example 1]{sa2}),
\begin{equation*}
P(-t \log |T'|) =\log  \sum_{i=1}^n |m_i|^t .
\end{equation*}
Hence, 
\begin{equation} \label{al}
\alpha(t)  =  -P(-t \log |T'|)' = \frac{\sum_{i=1}^n  |m_i|^t \log |m_i|}{\sum_{i=1}^n |m_i|^t}.
\end{equation}
The formula for $L(\alpha(t))$ follows from Equation \eqref{lt}.
\end{proof}

\begin{proof}[Proof of Corollary~\ref{C}]
From the  formula \eqref{al} for $\alpha(t)$, it follows that:
$$\sigma^2(t) = \frac{\sum_{i=1}^n  |m_i|^t (\log |m_i|)^2}{\sum_{i=1}^n |m_i| ^t} - 
\left(\frac{\sum_{i=1}^n  |m_i|^t \log |m_i|}{\sum_{i=1}^n |m_i| ^t }\right)^2 . $$
We introduce the following notation:

\begin{eqnarray*}
  \| \log|T'| \|^2_{2,t} & = & \frac{\sum_{i=1}^n  |m_i|^t (\log |m_i|)^2}{\sum_{i=1}^n |m_i| ^t}\\
 \| \log|T'| \|_{1,t} & = &\frac{\sum_{i=1}^n  |m_i|^t \log |m_i|}{\sum_{i=1}^n |m_i| ^t }
\end{eqnarray*}
Now from Equation~\eqref{sindividir}, we have that $L$ is concave if and only if, for all $t \in \mathbb{R}$, 

$$
2 P(-t \log |T'|) \left( \| \log|T'| \|^2_{2,t} - \| \log|T'| \|^2_{1,t} \right)  \le \| \log|T'| \|^2_{1,t}.$$
\end{proof}

\end{document}